\documentclass[number,citesort,dvips]{arxbj}
\usepackage{stfloats}

\aid{0}
\volume{17}
\issue{1}
\pubyear{2011}
\firstpage{441}
\lastpage{455}
\doi{10.3150/10-BEJ277}

\makeatletter

\fnbelowfloat

\newtheorem{theorem}{Theorem}
\newtheorem{lemma}[theorem]{Lemma}
\newtheorem{corollary}[theorem]{Corollary}
\newtheorem{prop}[theorem]{Proposition}
\newproclaim{defn}[theorem]{Definition}
\newremark{remark}[theorem]{Remark}

\newcommand{\eee}{\mathrm{e}}
\newcommand{\LL}{\mathcal{L}}
\newcommand{\real}{\mathbb{R}}
\newcommand{\Var}{\operatorname{\mathbb{V}ar}}

\newcommand{\BF}{\mathcal{BF}}
\newcommand{\CM}{\mathcal{CM}}
\newcommand{\CBF}{\mathcal{CBF}}
\newcommand{\St}{\mathcal{S}}

\renewcommand{\Re}{\operatorname{Re}}
\renewcommand{\Im}{\operatorname{Im}}
\newcommand{\eqref}[1]{(\ref{#1})}
\newcommand{\implies}{\Longrightarrow}
\makeatother

\begin{document}
\begin{frontmatter}

\title{From Schoenberg to Pick--Nevanlinna: Toward a complete picture
of the variogram class}
\runtitle{Pick--Nevanlinna variograms}

\begin{aug}
\author[a]{\fnms{Emilio} \snm{Porcu}\thanksref{a}\ead[label=e1]{eporcu@uni-goettingen.de}\corref{}} \and
\author[b]{\fnms{Ren\'{e} L.} \snm{Schilling}\thanksref{b}\ead[label=e2]{rene.schilling@tu-dresden.de}}
\runauthor{E. Porcu and R.L. Schilling}
\address[a]{Department of Mathematics, University Jaume I of Castell\'
{o}n, Campus Riu Sec, E-12071 Castell\'{o}n, Spain.
\printead{e1}}
\address[b]{Institute of Mathematical Stochastics, Technical University
Dresden, D-01062 Dresden, Germany.
\mbox{\printead{e2}}}
\end{aug}

% HISTORY:
\received{\smonth{12} \syear{2008}}
\revised{\smonth{3} \syear{2010}}

% ABSTRACT
%
\begin{abstract}
We show that a large subclass of variograms is closed under products
and that some desirable stability properties, such as the product of
special compositions, can be obtained within the proposed setting. We
introduce new classes of kernels of Schoenberg--L\'{e}vy type and
demonstrate some important properties of rotationally invariant
variograms.
\end{abstract}

% KEYWORDS
%
\begin{keyword}
\kwd{complete Bernstein functions}
\kwd{isotropy}
\kwd{Schoenberg--L\'{e}vy kernels}
\kwd{variograms}
\end{keyword}

\end{frontmatter}

%s1 ###
\section{Introduction}

Positive and conditionally positive definite functions on groups
or semigroups have a long history and appear in many
applications in probability theory, operator theory, potential
theory, moment problems and various other areas. They constitute an
important chapter in all treatments of harmonic analysis and their
origins can be traced back to papers by Carath\'eodory, Herglotz,
Bernstein and Matthias (see \cite{ber08} and references therein), culminating in Bochner's theorem from 1932;
see the surveys by Berg \cite{ber08} and Sasv\'{a}ri
\cite{sasvari08}. Schoenberg's theorem explains the possibility of
constructing rotationally invariant positive definite and
(the negatives of) conditionally positive definite functions on
Euclidean spaces via
completely monotone functions and Bernstein functions. Positive and
conditionally positive definite functions are a cornerstone of
spatial statistics where they are known, respectively, as
\textit{covariances} (or kernels) and \textit{variograms}. The theory of
random fields, which began in the 1940s with the early works of
Kolmogorov (see \cite{chidel} and references quoted therein) and
was further developed by Gandin \cite{GAN65} and Matheron
\cite{matheron73}, among others, is based on the specification of
these classes. In particular, the kriging predictor, that is to say,
the best linear unbiased predictor, depends exclusively on the
underlying covariance or variogram and we refer to the tour de force
in Stein \cite{stein99} for a rigorous assessment of this
framework.\looseness=-1

Let $\{Z(\xi), \xi\in\real^d \}$ be a stationary Gaussian random
field. The associated covariance function $C\dvtx\real^d
\to\real$ is positive definite, that is, for any finite collection
of points $\{\xi_i\}_{i=1}^n \in\real^d$, the matrix
$(C(\xi_i-\xi_j))_{i,j=1}^n$ is positive definite:\
\[
\mbox{for all } a_1, a_2, \ldots, a_n\in\mathbb{C}\qquad  \sum_{i,j=1}^n a_i
C(\xi_i-\xi_j) \overline{a}_j \geq0.
\]
Thus, a function $C\dvtx\real^d\to\real$ is positive definite if and
only if there exists a stationary Gaussian random field having $C(\cdot
)$ as
covariance function. If $C(\cdot)$ is \textit{rotationally invariant},
then the associated Gaussian random field is called \textit{isotropic}.

It is well known that the family of covariance functions is a convex
cone which is closed under products, pointwise convergence and scale
mixtures; for these basic facts, the reader is referred to standard
textbooks on geostatistics such as Chil\`es and Delfiner
\cite{chidel}.

A \textit{variogram} $\gamma\dvtx\real^d \to\real$ is the variance of the
increments of an intrinsically stationary random field, that is, for
any two points $\xi_1,\xi_2 \in\real^d$, $ \Var(Z(\xi_1)-Z(\xi_2)):=
\gamma(\xi_1-\xi_2)$.
Note that $\gamma(0)=0$, $\gamma(\xi)=\gamma(-\xi)$ and that
$-\gamma$
is \textit{conditionally positive definite}, that is, for any finite
collection of points $\{\xi_i\}_{i=1}^n \in\real^d$, we have
%
%e1 ###
\begin{equation}\label{eq:variogram}
\mbox{for all } a_1, \ldots, a_n \in\mathbb{C}\mbox{ such that}\qquad
\sum_{i=1}^n a_i = 0 ,\qquad
-\sum_{i,j=1}^n a_i \gamma(\xi_i-\xi_j)\overline{a}_j \geq0 .
\end{equation}
With a slight abuse of notation, we will also use the name
\textit{variogram} for a function $\gamma\dvtx\real^d\to\real$ with $\gamma
(0)\geq
0$ and such that $\gamma(\xi)-\gamma(0)$ is the variance of the
increments of an intrinsically stationary random field.

There is a close relationship between variograms $\gamma$ and
stationary covariance functions $C$. The elementary estimate
$|C(\xi)| \leq C(0)=:\Var Z$ shows that stationary covariance
functions are necessarily bounded; in particular, $\gamma(\xi) :=
C(0)-C(\xi)$ is a variogram. Indeed, variograms may be unbounded,
as in the case of fractional Brownian motion. If, however, the
variogram is bounded, then it is necessarily of the form
$C(0)-C(\xi)$, $\xi\in\real^d$, for some stationary covariance
function $C(\cdot)$; see, for instance, \cite{chidel} or \cite{berfor},
Proposition 7.13, and for a more general result due to Harzallah, see
\cite{Har1}.

The terminology concerning positive and conditional positive
definiteness is not uniform throughout the literature; it depends very
much on the mathematical context or the scientific application.
Christakos \cite{chris84} and many other applied scientists use the
notion of \textit{permissibility} for both concepts. We will use both
conventions alongside each other whenever no confusion can arise.

In this paper, we are mainly interested in rotationally invariant
covariances and variograms. This means that the associated Gaussian
random field is weakly or intrinsically stationary and isotropic.
Isotropy and
stationarity are independent assumptions, but we will assume both to
keep things simple. An isotropic covariance function, rescaled by
its value at the origin, is the characteristic function of a
rotationally symmetric random vector on the sphere of $\real^d$. This
class of covariances is well understood and we refer to Gneiting
\cite{gne00,gne01} and the references therein for an extensive
survey of this topic. Much less is known about variograms. For
instance, it is common knowledge that the class of variograms is a
convex cone which is closed in the weak topology of pointwise
convergence, but the product of two variograms is not necessarily a
variogram. This is a point that deserves a thorough discussion, in
the light of a recent beautiful result in \cite{ma07}, Theorem 3(i),
where a simple permissibility condition is given for the
product of two exponential variograms composed with a homogeneous
function.

We shall give a general answer to this question, as well as a
complete characterization of those variograms whose product is again
permissible. We shall then focus on
other challenging problems related to special compositions of
variograms, as well as to quasi-arithmetic compositions of them.

The use of kernels of Schoenberg--L\'{e}vy type has been persistently
emphasized in both old and recent literature. In this paper, we give
new forms of kernels of this type that may be appealing for
modeling in spatial statistics.

Another crucial problem faced in this paper regards the potential
trade-off between, on the one hand, the computational advantages
induced by the use of compactly supported kernels and, on the other
hand, the fact that compactly supported kernels can be positive
definite only on finite-dimensional spaces, by a striking and
beautiful result due to Wendland \cite{wen94}. We consider this
problem from the point of view of variograms; this makes sense since
variograms, which are possibly unbounded, represent a larger class
than covariance functions.

The paper is organized as follows. Section \ref{sec2} contains the basic
material required for a self-contained exposition and for
understanding the technical proofs of our statements. Section \ref{wow}
assesses new stability properties of the variogram class, while
Section \ref{sec4} is dedicated to kernels of Schoenberg--L\'{e}vy type.

%s2 ###
\section{Complete Bernstein functions and complete monotonicity}\label{sec2}

This section is mainly expository and we collect here some basic
material needed later. We will frequently use the following
characterization of variograms, for which a proof can be found in
\cite{berfor}, Proposition 7.5.

\begin{theorem}\label{cndf}
A function $\gamma\dvtx\real^d \to\real$ is a variogram if and only if
the following three conditions are satisfied:
\begin{longlist}[(iii)]
\item[(i)]$\gamma(0)\geq0$;
\item[(ii)]$\gamma(\xi)=\gamma(-\xi)$;
\item[(iii)]$-\gamma$ is conditionally positive definite, that is,\ equation
\eqref{eq:variogram} holds for all $\xi_1, \ldots, \xi_n\in\real^d$.
\end{longlist}
\end{theorem}

Let us remark that in harmonic analysis, functions satisfying
conditions (i)--(iii) of Theorem \ref{cndf} are often called \emph
{negative definite functions}. We will not use this notion in this paper.

Often, P\'{o}lya's theorem (see \cite{berfor}, Theorem 5.4) is useful
if one wants to construct concrete examples of variograms.
\begin{theorem}\label{polya}
A continuous function $\phi\dvtx\real\to[0,\infty)$ which is even
(i.e.,\
$\phi(x)=\phi(-x)$), decreasing and convex on the interval $(0,\infty)$
is positive definite.
\end{theorem}

Clearly, $\phi(0)-\phi(x)$ is increasing, concave and a variogram; see,
for example,\ \cite{berfor}, Corollary~7.7.

Recall that a function $f\dvtx(0,\infty)\to\real$ is called
\textit{completely monotone} if it is arbitrarily often
differentiable and
\[
(-1)^nf^{(n)}(x)\ge0\qquad \mbox{for } x>0, n=0,1,\ldots.
\]
By Bernstein's theorem, the set $\CM$ of
completely monotone functions coincides with the set of Laplace
transforms of positive measures $\mu$ on $[0,\infty)$, that is,\
\[\label{eq:bernstein}
f(x)=\LL\mu(x)=\int_{[0,\infty)} \eee^{-xt} \,\mathrm{d}\mu(t),
\]
where we only require that $\eee^{-xt}$ is $\mu$-integrable for any
$x>0$. $\CM$ is a convex cone which is closed under multiplication
and pointwise convergence.

\begin{defn}\label{d-stieltjes}
A function $f\dvtx(0,\infty)\to\mathbb R$ is called a \textup{Stieltjes
function} if it is of the form
%
%e2 ###
\begin{equation}\label{eq:stieltjes}
f(x)=a+\int_{[0,\infty)} \frac{\mathrm{d}\mu(t)}{x+t},
\end{equation}
where $a\ge0$ and $\mu$ is a positive measure on $[0,\infty)$ such
that $\int_{[0,\infty)} (1+t)^{-1}\, \mathrm{d}\mu(t)<\infty$.
\end{defn}

The following properties of the family $\St$ of Stieltjes functions
can be found in \cite{berfor}, Section~14, and \cite{ber08}. $\St$
is a convex cone such that $\St\subset\CM$. For every $f\in\St$, the
fractional power $f^\alpha\in\St\subset\CM$, $0<\alpha\le1$, is
again a Stieltjes function. Thus, for $f \in\St$, we see that
$f^{\alpha}$ is completely monotone for any $\alpha> 0$, so $f$
belongs to the set $\mathcal L$ of logarithmically completely
monotone functions discussed in, for example, \cite{ber08}, Section
2.6. The
formula
\[\label{eq:12}
\frac{1}{x(1+x^2)}=\int_{[0,\infty)} \eee^{-xt} (1-\cos t) \,\mathrm{d}t
\]
shows that $x^{-1}(1+x^2)^{-1}$ is completely monotone; however, it
cannot be a Stieltjes function since it has poles at $\pm i$ and
\eqref{eq:stieltjes} indicates that a Stieltjes function has a
holomorphic extension to the cut plane $\mathbb C\setminus
(-\infty,0]$. From the integral representation of $f$, it is
immediate that this extension satisfies $\Im z \Im f(z) \leq0$,
that is,\ $f$ maps the upper complex half-plane to the lower and
vice versa.

\begin{defn}\label{thm:defbern}
A function $f\dvtx(0,\infty)\to[0,\infty)$ is called a \textit{Bernstein
function} if it is infinitely often differentiable and $f'\in\CM$.
\end{defn}

The set of Bernstein functions is denoted $\BF$; it is a convex cone
which is closed under pointwise convergence. Since a Bernstein function is
non-negative and increasing, it has a non-negative limit $f(0+)$.
Integrating the Bernstein representation of the completely monotone
function $f'$ gives the following integral representation of
$f\in\BF$:
%
%e3 ###
\begin{equation}\label{eq:bernsteinfc}
f(x)=\alpha x+\beta+\int_{(0,\infty)} (1-\eee^{-xt}) \nu(\mathrm{d}t),
\end{equation}
where $\alpha,\beta\ge0$ are constants and $\nu$ is the
\textit{L\'{e}vy measure}, that is, a positive measure on
$(0,\infty)$ satisfying
\[
\int_{(0,\infty)} \frac{t}{1+t} \nu(\mathrm{d}t)<\infty.
\]

The following composition result will be useful throughout the paper;
see \cite{ber08}.
\begin{theorem}\label{thm:comp}
Let $\mathcal X$ be either of the sets $\BF, \CM$. Then
\[
f\in\mathcal X, g\in\BF\quad \implies\quad  f\circ g\in\mathcal X.
\]
\end{theorem}

If we assume that the representing measure $\nu(\mathrm{d}t)$ in
\eqref{eq:bernsteinfc} is of the form $\nu(\mathrm{d}t)= m(t) \,\mathrm{d}t$, where
$m(t)$ is completely monotone, then we get the family of \textit{complete
Bernstein functions}. We denote the collection of all complete
Bernstein functions by $\CBF$. It is not hard to see that $\CBF$ is,
like $\BF$, a convex cone which is closed under pointwise limits.
Complete Bernstein functions are widely used in various fields and
they are closely related to the following concepts: Bondesson
$T_2$-class (see \cite{bon} for the original definition and \cite
{berfor2} for further information),
operator-monotone functions (the classical source is \cite{don}) and
Pick functions (which are also known as Nevanlinna functions, i.e.,\
holomorphic functions in the upper half-plane with non-negative
imaginary part there).
A detailed survey can be found in \cite{ssv}, and short introductions in
\cite{schjams,jacbook3,ber08}. Among the most
prominent examples of complete Bernstein functions are
\begin{eqnarray*}
x&\mapsto&\frac{\lambda x}{\lambda+ x}\qquad  (\lambda>0),\qquad
x\mapsto x^\alpha\qquad (0<\alpha<1),\\
x&\mapsto&\log(1+x),\qquad
x\mapsto\sqrt x\arctan\frac1{\sqrt x}.
\end{eqnarray*}
Further examples are given below in Table \ref{table1}. Many Bernstein
functions given in closed form are already in $\CBF$. There are not
many known examples of functions in $\BF\setminus\CBF$ and they are
all finite or infinite sums of the form $\sum_i
p_i (1-\eee^{-\lambda_i x})$; see \cite{ber08}. Some interesting
examples are given in terms of the $q$-versions of the digamma function
$\psi_q(x)$ and Euler's constant $\gamma_q$: the function $x\mapsto
\psi_q(x+1)+\gamma_q$ is in $\BF\setminus\CBF$; see
\cite{krett}.\footnote{We are grateful to a referee supplying this
reference.}

%t1 ###
\begin{table}[b]
\caption{Examples of complete Bernstein functions
($\Gamma(a;x):=\int_x^\infty t^{a-1} \eee^{-t} \,\mathrm{d}t$ is the incomplete
Gamma function)}
\label{table1}
\begin{tabular*}{\textwidth}{@{\extracolsep{\fill}}ll@{\qquad\qquad}ll@{}}
\hline
Function & Parameter restrictions & Function & Parameters
restriction \\
\hline
$ 1-\frac1{(1+x^\alpha)^{\beta}}$ & $0<\alpha,\beta\leq1$ &
$ \eee x - x (1+\frac1x )^x - \frac x{x+1}$ & \\[5pt]
$ (\frac{x^\rho}{1+x^\rho} )^\gamma$ & $0<\gamma,\rho<1$ &
$ \frac1a - \frac1x \log(1+\frac xa )$ & $a>0$\\[5pt]
$ \frac{x^\alpha- x(1+x)^{\alpha-1}}{(1+x)^\alpha- x^\alpha}$ &
$0<\alpha< 1$ &
$ \sqrt{\frac x2} \frac{\sinh^2\sqrt{2x}}{\sinh(2\sqrt{2x})}$ &\\[5pt]
$ \sqrt x (1-\eee^{-2a\sqrt x} )$ & $a>0$ &
$ x^{1-\nu} \eee^{ax} \Gamma(\nu; ax)$ & $a>0, 0<\nu<1$\\[3pt]
$ \frac{x (1-\eee^{-2\sqrt{ x+a}} )}{\sqrt{x+a}}$ & $a>0$ &
$ x^\nu\eee^{a/x} \Gamma(\nu;\tfrac ax )$ & $a>0, 0<\nu<1$ \\
\hline
\end{tabular*}
\end{table}

The following statements are taken from \cite{ssv}.
\begin{theorem}\label{thm-cbf}
A function $f\dvtx(0,\infty)\to[0,\infty)$ such that $f(0+)$ exists is a
complete Bernstein function if and only if it has an analytic extension
to the cut complex plane $\mathbb{C}\setminus(-\infty,0]$ such that $\Im
z\cdot\Im f(z) \geq0$, that is, $f$ preserves upper and lower
half-planes. In particular, all complete Bernstein functions are of the form
%
%e4 ###
\begin{equation}\label{e-cbf}
f(z) = bz + a + \int_{(0,\infty)} \frac{z}{z+t} \sigma(\mathrm{d}t)
\end{equation}
with $a,b\geq0$ and a measure $\sigma$ satisfying $\int_{(0,\infty)}
(1+t)^{-1} \,\mathrm{d}t<\infty$.
\end{theorem}

Proofs of this classic result can also be found in \cite{jacbook3,schjams,ber08}. Theorem \ref{thm-cbf} can be used to
show that, for any $f\not\equiv0$,
%
%e5 ###
\begin{equation}\label{eq-cbf-1}
f\in\CBF\quad \iff\quad \biggl[x\mapsto\frac{f(x)}x \biggr] \in\St\quad \iff\quad \biggl[x\mapsto\frac
x{f(x)} \biggr] \in\CBF
\quad \iff\quad \frac1f\in\St.
\end{equation}
Let us briefly indicate the argument: if $f\in\CBF$, then we can use
\eqref{e-cbf} and divide by $z$. Comparing the resulting formula with
\eqref{eq:stieltjes} reveals that $f(z)/z$ is (the extension to
$\mathbb C\setminus(-\infty,0]$ of) a Stieltjes function. Therefore
(see\ the comment following Definition \ref{d-stieltjes}), we know that
$f(z)/z$ maps the upper to the lower complex half-plane. Consequently,
the inverse $g(z) := z/f(z)$ preserves upper and lower half-planes and
is, by Theorem \ref{thm-cbf}, in $\CBF$. Using the integral
representation \eqref{e-cbf} for $g$ and dividing by $z$, we get that
$g(z)/z = 1/f(z)$ is (the extension of) a Stieltjes function. As
before, we see that $f = 1/(1/f)$ preserves upper and lower half-planes
and is, therefore, a~complete Bernstein function. This proves all
equivalences in \eqref{eq-cbf-1}.

Using the fact that (the extensions of) functions in $\CBF$ preserve,
and those in $\St$ swap, complex half-planes, we immediately get the
following result. If we let $\mathcal X$ be either $\CBF$ or $\St$, then
\[\label{eq-cbf-2}
f, g\in\mathcal X\quad
\implies\quad  f\circ g\in\CBF.
\]
The following stability properties are less obvious.

\begin{theorem}\label{thm-cbf-stable}
Let $f, g, h\in\CBF$ and $f\not\equiv0$. Then:
\begin{longlist}[(iii)]
\item[(i)]
$ (f^\alpha(x)+g^\alpha(x))^{1/\alpha}\in\CBF$ for all $\alpha\in
[-1,1]\setminus\{0\}$;
\item[(ii)]
$ (f(x^\alpha)+g(x^\alpha) )^{1/\alpha}\in\CBF$ for all $\alpha
\in
[-1,1]\setminus\{0\}$;
\item[(iii)]
$ f(x^{\alpha})\cdot g(x^{1-\alpha})\in\CBF$ for all $\alpha\in[0,1]$;
\item[(iv)]
$ h(f(x))\cdot g (\frac{x}{f(x)} )\in\CBF$.
\end{longlist}
\end{theorem}

Assertion (iv) was discovered by Uchiyama \cite{uch06}, Lemma 2.1,
and since fractional powers $f(x)=x^\alpha$, $0\leq\alpha\leq1$, are
in $\CBF$, (iv) implies (iii). For positive $\alpha>0$, assertions
(i), (ii) are in \cite{nak89} -- his proofs are easily
adapted to $\alpha< 0$ since $f\in\CBF$ if and only if $1/f\in
\St$; see \eqref{eq-cbf-1}. A~unified approach will appear in \cite{ssv}.

Letting $\alpha\to0$ in Theorem \ref{thm-cbf-stable} proves $\lim
_{\alpha\downarrow0} (\tfrac12 f^\alpha+ \tfrac12 g^\alpha
)^{1/\alpha}=\sqrt{fg}$ and since pointwise limits of complete
Bernstein functions are complete Bernstein, we see that $\sqrt{fg}\in
\CBF$ whenever $f,g\in\CBF$. From this, we can easily deduce a new
proof of the so-called log-convexity of the convex cone $\CBF$:
%
%e6 ###
\begin{equation}\label{eq:log-convex}
f,g\in\CBF, \alpha\in[0,1]\quad
\implies\quad  f^\alpha\cdot g^{1-\alpha}\in\CBF.
\end{equation}
Alternative proofs can be found in \cite{ber79} and \cite{ssv}.

Indeed, if $\alpha$ is a dyadic number of the form $\alpha=
\sum_{i=1}^n \alpha_i 2^{-i}$ with $\alpha_i\in\{0,1\}$ and
$\alpha_n=1$, then $\alpha'=1-\alpha$ is of the same type with
$\alpha_n'=1$. This is because
\[
\alpha^{\prime}
= \sum_{i=1}^{\infty} 2^{-i} -\sum_{i=1}^{n} \alpha_i 2^{-i}
= \sum_{i=1}^{n-1}(1-\alpha_i)2^{-i} +\sum_{i=n+1}^{\infty}2^{-i}
= \sum_{i=1}^{n-1}\alpha_i^{\prime} 2^{-i} + 2^{-n}
\]
with $\alpha_i^{\prime}=1-\alpha_i$, $i=1,\ldots,n-1$. This means
that
\[
f^\alpha g^{1-\alpha}
= \prod_{i=1}^n \sqrt[2^i]{f^{\alpha_i} g^{\alpha_i'}}
= \sqrt{h_1\sqrt{ h_2 \cdots\sqrt{h_{n-2}\sqrt{ h_{n-1}\sqrt{f_n
g_n }}}}},
% = \sqrt{\sqrt{\sqrt{\sqrt{\sqrt{f_n g_n } h_{n-1} } h_{n-2} } h_{n-3}
%} \cdots h_1 }
\]
where $h_i$ stands for either $f_i$ or $g_i$ if $\alpha_i = 1$ or
$\alpha_i=0$, respectively. Thus, repeated applications of \eqref
{eq:log-convex} with $\alpha=\alpha'=\frac12$ lead to \eqref
{eq:log-convex} for all dyadic $\alpha\in(0,1)$. Since $(0,1)\ni
\alpha
\mapsto f^\alpha$ is continuous, we get \eqref{eq:log-convex} for all
$\alpha\in(0,1)$.

%s3 ###
\section{Variograms and their stability properties}\label{wow}%wow =
%WOnderful World

As already emphasized in Section 1, the starting point for this work
is a result in \cite{ma07}, Theorem~3(i), which is reported below
with a short alternative proof.

\begin{theorem}[(\cite{ma07})]\label{ma2007}
Let $\gamma\dvtx \real^d \to\real$ be a homogeneous function. Then
%
%e7 ###
\begin{equation}\label{eq:Ma2007}
\bigl( 1- \eee^{-a_1 \gamma(\xi)} \bigr) \bigl( 1- \eee^{-a_2 \gamma(\xi
)} \bigr),
\end{equation}
$a_i > 0$, $i=1,2$, is a variogram if and only if $\gamma(\xi)=|A\xi|$
for the Euclidean norm $|\cdot|$ and a $d\times d$ matrix $A$.
\end{theorem}

It is natural to ask whether Ma's theorem works only for the
exponential class of variograms or whether it can be generalized.
The subsequent result gives an answer to this problem, supplying a
wide class of variograms closed under products.

Here and hereafter, we will use a famous result of Schoenberg and
Bochner; see \cite{schoe38} (in the context of covariance functions and
complete monotonicity) and  \cite{bochner}, page 99 (in the
context of variograms and Bernstein functions). We restate Bochner's
version in the setting of the current paper. Alternative proofs can be
found in the Appendix of Jacob and Schilling \cite{jacsch05} and
Steerneman and van-Perlo-ten Kleij \cite{SP}.

\begin{lemma}\label{lem-scho}
All variograms $\gamma$ which are rotationally invariant and
permissible in all (or at least infinitely many) dimensions
$d=1,2,\ldots$ are of the form $\gamma(\xi)=f(|\xi|^2)$ with a
Bernstein function $f\in\BF$.
\end{lemma}

The next result is not only a generalization of Ma's result, but also
the key to a simple proof of Theorem \ref{ma2007}.
\begin{theorem}\label{thm-ma-extended-1}
Let $g_1,g_2$ be Bernstein functions and $0 \leq\alpha_1, \alpha_2$
such that $\alpha_1 + \alpha_2 \leq1$. Then
$g_1(x^{\alpha_1})g_2(x^{\alpha_2})$ is a Bernstein function.
\end{theorem}

\begin{pf}
%The proof of this result follows the arguments in Schilling, Song
%and Vondra\v{c}ek \cite{ssv}. Indeed, if $\gamma_i$ are rotationally
%invariant and permissible on any $d$-dimensional Euclidean space,
%then by Lemma \ref{lem-scho} they must be of the form
%$\gamma_{i}(\xi):=g_i(|\xi|^2)$, where $g_i \in\BF$.
%
Set $h_{\alpha,\beta}(x) := g_1(x^\alpha) \cdot g_2(x^\beta)$,
$x>0$. It is enough to show that $h_{\alpha,\beta}'\in\CM$.
%Indeed,
%this proves that $h_{\alpha,\beta}$ is a Bernstein function and that
%$f_{\alpha,\beta}(\xi) = h_{\alpha,\beta}(|A \xi|^2)$ is a
%variogram.
%
%Elementary algebra shows that
Clearly,
\[
h'_{\alpha,\beta}(x)= x^{\alpha+\beta-1} \biggl(\alpha g_1'(x^{\alpha})
\frac
{g_2(x^{\beta})}{x^{\beta}}
+ \beta g_2'(x^{\beta}) \frac{g_1(x^{\alpha})}{x^{\alpha}} \biggr).
\]
Since $g_i \in\BF$, we have that $g_i^{\prime} \in\CM$ and
$x^{-1}g_i(x) \in\CM$. This will also be the case for the compositions
$g_1'(x^{\alpha})$ and $g_2'(x^{\beta})$,
$g_1(x^{\alpha})/x^{\alpha}$ and $g_2(x^{\beta})/x^{\beta}$, by a
straightforward application of Theorem \ref{thm:comp}. Moreover, for
$\alpha+\beta\leq1$, $x\mapsto x^{\alpha+\beta-1}$ is completely
monotone. The proof is completed since completely monotone functions
form a convex cone which is closed under products.
\end{pf}

\begin{corollary}\label{thm-ma-extended}
Let $\real^d \ni\xi\mapsto\gamma_i(\xi) = g_i(|\xi|^2)$ be
rotationally invariant variograms for all $d \in\mathbb{N}$,
$i=1,2$. Let $\alpha, \beta\in[0,1]$ be such that $\alpha+ \beta
\leq1$ and let $A$ be a $d\times d$ matrix. Then
\[
f_{\alpha,\beta}(\xi):= g_1(|A \xi|^{2\alpha}) g_2(|A \xi
|^{2\beta})
\]
is still a variogram on $\real^d$ for all $d\in\mathbb N$.
\end{corollary}

\begin{remark}
The result of Theorem \ref{thm-ma-extended-1} extends immediately to
the product of $n$ Bernstein functions: for $\sum_{i=1}^n \alpha_i
\leq
1$, $\alpha_i\geq0$ and $g_i\in\BF$, the
function $h(x) := \prod_{i=1}^n g_i(x^{\alpha_i})$ is again in $\BF$.
This generalizes the case where $\alpha_i = \frac1n$, $g_i=g$,
$i=1,\ldots,n$, leading to $h(x) = (g(x^{1/n}))^n$, which is due to
\cite{BBL}.
\end{remark}

The proof of the result above offers a considerably easier way to
show Ma's result.

\begin{pf*}{Proof of Theorem \ref{ma2007}}
If $\gamma(\xi)=|A\xi|$, Corollary \ref{thm-ma-extended} with
$g_i(x)=1-\exp(-a_i x)$, $i=1,2$ and $\alpha=\beta=\frac12$ shows that
\eqref{eq:Ma2007} is a variogram.

Now, assume that \eqref{eq:Ma2007} is a variogram. Then
\[
\xi\mapsto\frac{ ( 1- \eee^{\it-a_1 \gamma(\xi)} )}{a_1}
\cdot\frac{ ( 1- \eee^{\it-a_2 \gamma(\xi)} )}{a_2}
% \xrightarrow[ a_2\to0 ]{a_1\to0}
% \gamma^2(\xi)
\]
is a variogram for all $a_1, a_2>0$ and so is its pointwise limit
$\gamma^2(\xi)$ as $a_1,a_2\to0$; thus, $\gamma^2(\xi)$ is a
real-valued variogram. As such, it has a L\'evy--Khinchine representation
\[
\gamma^2(\xi)
= Q\xi\cdot\xi+ \int_{x\neq0} \bigl(1-\cos
(x\cdot\xi)\bigr) \nu(\mathrm{d}x),
\]
where $Q\in\real^{d\times d}$ is positive semi-definite and $\nu$ is a
measure with $\int_{x\neq0} |x|^2\wedge1 \nu(\mathrm{d}x) < \infty$. Since
$\gamma(\xi)$ is homogeneous, we get
\[
\gamma^2(\xi) = \frac{\gamma^2(n\xi)}{n^2} \stackrel{n\to
\infty}{\longrightarrow} Q\xi
\cdot\xi= \bigl|\sqrt Q \xi\bigr|^2
\]
for the uniquely determined, positive semidefinite square root $A=\sqrt
Q$ of $Q$.
\end{pf*}

Several examples of Bernstein functions may be found in
\cite{berfor,ber08} or in \cite{jacsch05}; an extensive
list will be included in the monograph \cite{ssv}. Three celebrated
classes of Bernstein functions are well known in the spatial
statistics literature:
\begin{enumerate}[(3)]
\item[(1)] the \textit{Mat\'{e}rn} class \cite{matern1986}
$ f_{\alpha,\nu} = 1- 2^{1-\nu}/\Gamma(\nu) (\alpha\sqrt
x)^\nu K_\nu(\alpha\sqrt x) $, $x>0$, for $\alpha,
\nu
> 0$ and where $K_{\nu}$ is the modified Bessel function
of the second kind of order $\nu$;

\item[(2)] the \textit{Cauchy} class \cite{gneischlat2004}
$ f_{\alpha,\beta}(x):= 1- (1+x^{\alpha})^{-\beta} $, $x>0$, where $0
< \alpha\leq1$ and $0 < \beta$;

\item[(3)] the \textit{Dagum} class \cite{porcu}
$ f_{\rho,\gamma}(x):= ( \frac{x^{\rho}}{1+x^{\rho}} )^{\gamma}$,
$x>0$, where $\rho,\gamma\in(0,1)$.
\end{enumerate}

Let us mention a few more stability properties that make some
classes of functions appealing for their use in spatial statistics.
We again work within the framework of rotationally invariant
functions.

\begin{prop}
Let $\gamma\dvtx \real^d \to\real$ be rotationally invariant for all
dimensions $d=1,2,\ldots$ such that $\gamma(\xi)=g(|\xi|^2)$ for some
$g \in
\CBF$. Then:
\begin{longlist}[(ii)]
\item[(i)]
$ \real^d \ni\xi\mapsto\frac{|\xi|^2}{g(|\xi|^2)}$ is a
rotationally invariant variogram which is permissible for every $d \in
\mathbb N$;

\item[(ii)]$ \real^d \ni\xi\mapsto\frac1{g (1/{|\xi|^2} )}$ and
$\xi\mapsto|\xi|^2 g (\frac1{|\xi|^2} )$ are rotationally invariant
variograms which are permissible for every $d \in\mathbb N$.
\end{longlist}
\end{prop}

\begin{pf}
Part (i) is a simple application of the first equivalence in
\eqref{eq-cbf-1} which states that $g \in\CBF$ if and only if
$g(x)/x$ is a Stieltjes function.

Part (ii) follows immediately by noting that, for $g \in\CBF
$, the function $x \mapsto1/ g(1/x)$ is a composition of the type
$\sigma\circ g \circ\sigma(x)$, where $\sigma$ is the Stieltjes
function $x \mapsto\frac1x$. Since the composition $\sigma\circ g$
is a Stieltjes function and since the composition of two Stieltjes
functions is in $\CBF$, we have the first assertion of part
(ii). If we apply part (i) to this variogram, the second
assertion follows.
\end{pf}

For further (stability) properties of the class $\CBF$, the reader is
referred to \cite{ssv}; some examples of complete Bernstein functions
are given below.

Another interesting problem arises when quasi-arithmetic operators,
in the sense of Hardy, Littlewood and P\'olya \cite{harlitpol},
are applied to variograms. This means that we seek conditions which
preserve the permissibility of the underlying structure. This has
been considered in \cite{pormatchr07} for quasi-arithmetic
composition of covariance functions. We believe that the same
question in connection with variograms is even more challenging from
the mathematical point of view and is equally important as far as
statistics are concerned.

Recall that a \textit{power mean} is a mapping of the form $(u,v)
\mapsto\psi_{\alpha}(u,v):= ( u^{\alpha}+v^{\alpha}
)^{1/\alpha}$ for $(u,v) \in\real^2$ and $\alpha\in\real
\setminus\{0\}$.
\begin{prop}\label{rene1}
Let $\gamma_i\dvtx\real^d \to\real$, $i=1,2$, be rotationally invariant
variograms for all dimensions $d \in\mathbb N$. We write $g_i$ for the
radial function such that $\gamma_i(\xi)=g_i(|\xi|^2)$:
\begin{longlist}[(iii)]
\item[(i)] If $g_1,g_2\in\CBF$, then
$ \xi\mapsto(\gamma_1^\alpha(\xi) + \gamma_2^\alpha(\xi)
)^{1/\alpha
}$ is a variogram for all $\alpha\in[-1,1] \setminus\{0\}$.
\item[(ii)] If $g_1,g_2\in\CBF$, then
$ \xi\mapsto(g_1(|\xi|^{2\alpha}) + g_2(|\xi|^{2\alpha} )
)^{1/\alpha}$ is a variogram for all $\alpha\in[-1,1] \setminus\{0\}$.
\item[(iii)]
$ \xi\mapsto g_1(|\xi|^{2\alpha}) g_2(|\xi|^{2-2\alpha})$ is a
variogram for all $0 < \alpha< 1$.
\end{longlist}
\end{prop}

\begin{pf}
Since, by Lemma \ref{lem-scho}, $g_i\in\BF$, assertion (iii) is a
simple consequence of Corollary \ref{thm-ma-extended}. We should
mention at this point that for $g_1, g_2\in\CBF$, the resulting
rotationally invariant variogram would again be of the form $h(|\xi
|^2)$ with $h\in\CBF$; see Theorem \ref{thm-cbf-stable}(iv).
Both (i) and (ii) follow immediately from \ref{thm-cbf-stable}(i) and
(ii), respectively.
\end{pf}

Finally, we combine two aspects treated separately until now. Given
two or three isotropic variograms, we seek permissibility conditions
for the products of special compositions. The proposition below
results from a simple application of Theorem \ref{thm-cbf-stable}(iv)
with $h(s) = s$, $f = g_1$, $g = g_2$, respectively, $h =
g_3$, $f = g_1$, $g = g_2$.

\begin{prop}
Let $\real^d \ni\xi\mapsto\gamma_i(\xi)$, $i=1,2,3$, be
rotationally invariant and isotropic variograms for all $d \in\mathbb
{N}$ and assume that $\gamma_i(\xi)=g_i(|\xi|^2)$, where $g_i\in
\CBF$. Then
\[
\xi\mapsto\gamma_1(\xi) \gamma_2 \biggl( \frac{\xi}{\sqrt{\gamma
_1(\xi)}} \biggr)
\quad \mbox{and}\quad
\gamma_3 \bigl( \sqrt{\gamma_1(\xi)} \bigr) \gamma_2 \biggl(\frac{\xi}{\sqrt
{\gamma
_1(\xi)}} \biggr)
\]
are still permissible for all $d \in\mathbb N$ and of the form $h(|\xi
|^2)$ with some $h\in\CBF$.
\end{prop}

We conclude this section by presenting another curious way to construct
continuous variograms and, more generally, complex-valued conditionally
positive definite functions, with the help of Bernstein functions. The
interesting fact in the example below is the product structure, which
is quite unusual for conditionally positive definite functions.

\begin{prop}
Let $f$ be a Bernstein function such that the representing measure $\nu
$ in the L\'evy--Khinchine formula \eqref{eq:bernsteinfc} has a
monotone decreasing density $m$, that is,\ $f(x)=\alpha+\beta x+\int
_{(0,\infty)} (1-\eee^{-xt}) m(t) \,\mathrm{d}t$.

Then $\xi\mapsto \mathrm{i}\xi f(\mathrm{i}\xi)$ is conditionally positive definite and
$\xi\mapsto-\Re(\mathrm{i}\xi f(\mathrm{i}\xi))$ is a continuous variogram.
\end{prop}

\begin{pf}
By the monotonicity of $m$, we see that $m(t) = \nu[t,\infty)$ for a
(L\'evy) measure $\nu$, that is, a~measure $\nu$ on $(0,\infty)$
satisfying $\int_{(0,\infty)} t(1+t)^{-1} \nu(\mathrm{d}t)$. The integration
properties of $\nu$ become clear from the calculation below since we
have only used Fubini's theorem for positive integrands to swap
integrals. For $x\geq0$, we get
\begin{eqnarray*}
xf(x)
&=& \alpha x + \beta x^2 + \int_0^\infty x(1-\eee^{-xt})\int_t^\infty
\nu(\mathrm{d}s) \,\mathrm{d}t\\
&=& \alpha x + \beta x^2 + \int_0^\infty\int_t^\infty x(1-\eee^{-xt})
\nu(\mathrm{d}s) \,\mathrm{d}t\\
&=& \alpha x + \beta x^2 + \int_0^\infty\int_0^s x(1-\eee^{-xt}) \,\mathrm{d}t
\nu
(\mathrm{d}s)\\
&=& \alpha x + \beta x^2 + \int_0^\infty[\eee^{-xs} -1+ sx ] \nu(\mathrm{d}s),
\end{eqnarray*}
which, as by-product, shows that $\int_0^\infty s^2\wedge s \nu(\mathrm{d}s) <
\infty$. We may, therefore, plug in $z=\mathrm{i}\xi$ and get
\[
\mathrm{i}\xi f(\mathrm{i}\xi)
= - \biggl(-\mathrm{i}\alpha\xi+ \beta\xi^2 + \int_0^\infty[1-\eee^{-\mathrm{i}s\xi} -
\mathrm{i}s\xi] \nu(\mathrm{d}s) \biggr).
\]
Thus, $-\gamma(\xi) := \mathrm{i}\xi f(\mathrm{i}\xi)$ is conditionally positive definite
and $\Re\gamma(\xi)$ is a variogram.
\end{pf}

%s4 ###
\section{Kernels and variograms of the Schoenberg--L\'{e}vy type}\label{sec4}

This section explores some results that may be obtained when working
with kernels of the Schoenberg--L\'{e}vy type. These kernels are
extensively used in the literature and we refer to Ma \cite{ma07}
and the references therein. For $\xi_1,\xi_2\in\real^d$, these
are non-stationary covariance functions obtained from a non-negative
function $g\dvtx[0,\infty) \to[0,\infty)$ such that $g(0)=0$ through
the linear combination
\[
g(|\xi_1|)+g(|\xi_2|) - g(|\xi_1-\xi_2|).
\]
A celebrated example is the fractional Brownian sheet \cite{ayache}
with $g(\xi)=|\xi|^{\alpha}$, $\alpha\in(0,2]$. Ma \cite{ma07}
shows that for a fixed $\xi_0\in\real^d$, the function
\[
C_{\xi_0}(\xi):= g(|\xi+\xi_0|)+g(|\xi-\xi_0|)-2 g (|\xi|)
\]
is a covariance function, provided that $g(|\xi|)$ is a variogram.
Indeed, we are going to show that this is a simple consequence of
the following, more general, result.
\begin{lemma}\label{lemma-schoenberg-kernel}
Let $\gamma\dvtx \real^d \to\real$ be a continuous variogram and let
$\xi
, \eta\in\real^d$, $d \in\mathbb{N}$. Then
\[
\phi_\eta(\xi) := \gamma(\xi+\eta)+\gamma(\xi-\eta)-2\gamma
(\xi)
\]
is a continuous covariance function as a function of $\xi$. Moreover, if
\[
\gamma_\eta(\xi) := 2\gamma(\eta)+ 2\gamma(\xi) - \gamma(\xi
+\eta
)-\gamma(\xi-\eta),
\]
then $\xi\mapsto\gamma_\eta(\xi)$ is a continuous variogram.
\end{lemma}

Note that in Lemma \ref{lemma-schoenberg-kernel}, we have $\gamma
_\eta
(\xi)=\gamma_\xi(\eta)$, that is, $\eta\mapsto\gamma_\eta(\xi
)$ is also
a continuous variogram.

\begin{pf*}{Proof of Lemma \ref{lemma-schoenberg-kernel}}
Recall the following elementary formula for the cosine:
$
\cos(a+b)+\cos(a-b) = 2\cos a \cos b.
$
Since $\gamma(\xi)$ has the L\'evy--Khinchine representation
\[
\gamma(\xi) = Q\xi\cdot\xi+ \int_{x\neq0} (1-\cos x\cdot\xi)
\nu(\mathrm{d}x),
\]
we find that
\begin{eqnarray*}
\phi_\eta(\xi)
&=& Q(\xi+\eta)\cdot(\xi+\eta) + Q(\xi-\eta)\cdot(\xi-\eta) -
2Q\xi\cdot
\xi\\
&&{} + \int_{x\neq0} \bigl(2\cos x\cdot\xi- \cos
x\cdot
(\xi+\eta) -\cos x\cdot(\xi-\eta) \bigr) \nu(\mathrm{d}x)\\
&=& 2Q\eta\cdot\eta+ \int_{x\neq0} (2\cos x\cdot\xi- 2\cos x\cdot
\xi
\cos x\cdot\eta) \nu(\mathrm{d}x)\\
&= &2Q\eta\cdot\eta+ 2\int_{x\neq0} (1 - \cos x\cdot\eta) \cos
x\cdot
\xi\nu(\mathrm{d}x).
\end{eqnarray*}
This shows that $\xi\mapsto\phi_\eta(\xi)$ is symmetric and positive
definite, hence a covariance function. Now, consider
\begin{eqnarray*}
\gamma_\eta(\xi)
&=& 2\gamma(\eta)-\phi_\eta(\xi)\\
&=& 2a + 2Q\eta\cdot\eta+ 2\int_{x\neq0} (1 - \cos x\cdot\eta)
\nu(\mathrm{d}x)\\
&&{}
- 2Q\eta\cdot\eta- 2\int_{x\neq0} (1 - \cos x\cdot\eta) \cos
x\cdot
\xi\nu(\mathrm{d}x)\\
&=& 2a + 2\int_{x\neq0} (1 - \cos x\cdot\eta) (1-\cos
x\cdot\xi) \nu(\mathrm{d}x).
\end{eqnarray*}
Thus, $ \gamma_\eta(\xi)$ is a variogram in $\xi$. The proof is
thus complete.
\end{pf*}

Lemma \ref{lemma-schoenberg-kernel} has an obvious extension to
continuous \textit{complex}-valued functions $\gamma\dvtx\real^d\to
\mathbb{C}$ satisfying $\gamma(0)\geq0$,
$\gamma(\xi)=\overline{\gamma(-\xi )}$ and the permissibility condition
\eqref{eq:variogram} for all $\xi_1, \ldots , \xi_n\in\real^d$. Since
such functions also enjoy a (complex) L\'evy--Khinchine representation
(see \cite{berfor}), exactly the same argument as in the proof of Lemma
\ref{lemma-schoenberg-kernel} shows that for every fixed
$\xi_0\in\real^d$,
\[
\gamma_{\xi_0}(\xi) := 2\gamma(\xi) + 2\Re\gamma(\xi_0) -
\gamma(\xi
-\xi_0) - \gamma(\xi+\xi_0)
\]
is permissible and has the L\'evy--Khinchine representation
\[
\gamma_{\xi_0}(\xi) = 2\int_{y\neq0} (1-\eee^{\mathrm{i}y\cdot\xi} )
\bigl(1-\cos
(y\cdot\xi_0) \bigr) \nu(\mathrm{d}y),
\]
where $\nu$ is the L\'{e}vy measure of $\gamma$. Lemma 17 is a very
special case of \cite{berfor}, Proposition 18.2, which goes back to
Harzallah \cite{Har2}.

\section*{Acknowledgements}
We are grateful to three anonymous referees for their valuable
comments which helped to improve the presentation of this paper. This
work was initiated when the first-named author was visiting the
Technical University of Dresden. He is grateful to Professor Zolt\'{a}n
Sasv\'{a}ri for the invitation and the hospitality he was shown.
Emilio Porcu acknowledges the DFG-SNF Research Group FOR916,
subproject A2.

\printhistory

\end{document}